\documentclass[11pt,reqno]{amsart}

\usepackage[a4paper,hmargin={2.25cm,2.25cm},vmargin={2cm,2cm},heightrounded,
marginparwidth=2.1cm, marginparsep=0.05cm]{geometry}

\usepackage[utf8]{inputenc}
\usepackage[T1]{fontenc}
\usepackage{lmodern}

\usepackage{amssymb}
\usepackage[shortlabels]{enumitem}
\usepackage[all,arc,cmtip]{xy}
\usepackage[mathscr]{euscript}
\usepackage{stmaryrd} 
\usepackage{mathtools}

\theoremstyle{plain}
\newtheorem{theorem}{Theorem}[section]
\newtheorem{lemma}[theorem]{Lemma}
\newtheorem{proposition}[theorem]{Proposition}
\newtheorem{corollary}[theorem]{Corollary}

\theoremstyle{definition}

\newtheorem{examples}[theorem]{Examples}
\newtheorem{assumptions}[theorem]{Assumptions}

\theoremstyle{remark}
\newtheorem{remark}[theorem]{Remark}

\newlist{tfae}{enumerate}{1}
\setlist[tfae,1]{label=(\roman*)}

\DeclareMathOperator{\ev}{ev}

\DeclareMathOperator{\Sup}{Sup}
\DeclareMathOperator{\can}{can}

\newcommand{\catfont}[1]{\mathbf{#1}}

\newcommand{\catC}{\catfont{C}}
\newcommand{\catA}{\catfont{A}}
\newcommand{\catB}{\catfont{B}}
\newcommand{\App}{\catfont{App}}

\newcommand{\Met}{\catfont{Met}}
\newcommand{\Set}{\catfont{Set}}
\newcommand{\Top}{\catfont{Top}}
\newcommand{\Ord}{\catfont{Ord}}
\newcommand{\MultiOrd}{\catfont{MultiOrd}}

\newcommand{\Rels}[1]{#1\text{-}\catfont{Rel}}
\newcommand{\Cats}[1]{#1\text{-}\catfont{Cat}}
\newcommand{\Gphs}[1]{#1\text{-}\catfont{Gph}}

\newcommand{\V}{V}
\newcommand{\W}{W}

\newcommand{\two}{2}

\newcommand{\op}{\mathrm{op}}

\newcommand{\sep}{\mathrm{sep}}

\newcommand{\relto}{\mathrel{\mathmakebox[\widthof{$\xrightarrow{\rule{1.45ex}{0ex}}$}]{\xrightarrow{\rule{1.45ex}{0ex}}\hspace*{-2.4ex}{\mapstochar}\hspace*{1.8ex}}}}

\makeatletter

\def\slashedarrowfill@#1#2#3#4#5{%
  $\m@th\thickmuskip0mu\medmuskip\thickmuskip\thinmuskip\thickmuskip
   \relax#5#1\mkern-7mu%
   \cleaders\hbox{$#5\mkern-2mu#2\mkern-2mu$}\hfill
   \mathclap{#3}\mathclap{#2}%
   \cleaders\hbox{$#5\mkern-2mu#2\mkern-2mu$}\hfill
   \mkern-7mu#4$%
}

\newcommand*{\rightrelarrowfill@}{\slashedarrowfill@\relbar\relbar{\raisebox{0pc}{$\mapstochar$}}\rightarrow}

\newcommand*{\xrelto}[2][]{\ext@arrow 0055{\rightrelarrowfill@}{\;#1\;}{\;#2\;}}

\newcommand*{\rightmodarrowfill@}{\slashedarrowfill@\relbar\relbar{\raisebox{0pc}{$\hspace{0.3em}\circ$}}\rightarrow}

\newcommand*{\xmodto}[2][]{\ext@arrow 0055{\rightmodarrowfill@}{\;#1\;}{\;#2\;}}

\newcommand*{\rightkrelarrowfill@}{\slashedarrowfill@\relbar\relbar{\raisebox{0pc}{$\hspace{0.3em}\mapstochar$}}\rightharpoonup}

\newcommand*{\xkrelto}[2][]{\ext@arrow 0055{\rightkrelarrowfill@}{\;#1\;}{\;#2\;}}

\newcommand*{\rightkmodarrowfill@}{\slashedarrowfill@\relbar\relbar{\raisebox{0pc}{$\hspace{0.3em}\circ$}}\rightharpoonup}

\newcommand*{\xkmodto}[2][]{\ext@arrow 0055{\rightkmodarrowfill@}{\;#1\;}{\;#2\;}}

\makeatother

\newcommand{\fx}{\mathfrak{x}}
\newcommand{\fy}{\mathfrak{y}}
\newcommand{\fp}{\mathfrak{p}}
\newcommand{\fq}{\mathfrak{q}}
\newcommand{\fv}{\mathfrak{v}}
\newcommand{\fw}{\mathfrak{w}}
\newcommand{\fz}{\mathfrak{z}}

\newcommand{\fX}{\mathfrak{X}}
\newcommand{\fP}{\mathfrak{P}}
\newcommand{\fQ}{\mathfrak{Q}}

\newcommand{\mId}{\mathbb{I}}
\newcommand{\mH}{\mathbb{H}}
\newcommand{\mT}{\mathbb{T}}
\newcommand{\mU}{\mathbb{U}}
\newcommand{\mW}{\mathbb{W}}
\newcommand{\TV}{(\mT,\V)}
\DeclareMathAlphabet{\mathpzc}{OT1}{pzc}{m}{it}
\DeclareMathOperator{\yoneda}{\mathpzc{y}}

\newcommand{\monad}{(T,m,e)}

\usepackage[pdftex]{hyperref}

\author[M.M.~Clementino]{Maria Manuel Clementino}

\address{CMUC, Department of Mathematics, University of Coimbra, 3001-501 Coimbra,
  Portugal}

\email{mmc@mat.uc.pt}

\author[D.~Hofmann]{Dirk Hofmann}

\address{CIDMA, Department of Mathematics, University of Aveiro, 3810-193
  Aveiro, Portugal}

\email{dirk@ua.pt}

\author{Willian Ribeiro}

\address{CMUC, Department of Mathematics, University of Coimbra, 3001-501
  Coimbra, Portugal}

\email{willian.ribeiro.vs@gmail.com}

\thanks{Research partially supported by Centro de Matem\'{a}tica da Universidade
  de Coimbra -- UID/MAT/00324/2019, by Centro de Investiga\c{c}\~ao e
  Desenvolvimento em Matem\'atica e Aplica\c{c}\~oes da Universidade de
  Aveiro/FCT -- UID/MAT/04106/2019, funded by the Portuguese Government through
  FCT/MCTES and co-funded by the European Regional Development Fund through the
  Partnership Agreement PT2020. W. Ribeiro also acknowledges the FCT PhD Grant
  PD/BD/128059/2016.}



\begin{document}

\title{Cartesian closed exact completions in topology}

\begin{abstract}
  Using generalized enriched categories, in this paper we show that
  Rosick\'{y}'s proof of cartesian closedness of the exact completion of the
  category of topological spaces can be extended to a wide range of topological
  categories over $\Set$, like metric spaces, approach spaces, ultrametric
  spaces, probabilistic metric spaces, and bitopological spaces. In order to do
  so we prove a sufficient criterion for exponentiability of $\TV$-categories
  and show that, under suitable conditions, every injective $\TV$-category is
  exponentiable in $\Cats{\TV}$.
\end{abstract}

\maketitle

\section{Introduction}

As Lawvere has shown in his celebrated paper \cite{Law73}, when $\V$ is a closed
category the category $\Cats{\V}$ of $\V$-enriched categories and $\V$-functors
is also monoidal closed. This result extends neither to the cartesian structure
nor to the more general setting of $\TV$-categories. Indeed, cartesian
closedness of $\V$ does not guarantee cartesian closedness of $\Cats{\V}$: take
for instance the category of (Lawvere's) metric spaces $\Cats{P_+}$, where $P_+$
is the complete real half-line, ordered with the $\geq$ relation, and equipped
with the monoidal structure given by addition $+$; $P_+$ is cartesian closed but
$\Cats{P_+}$ is not (see \cite{CH06} for details); and, even when the monoidal
structure of $\V$ is the cartesian one, the category $\Cats{\TV}$ of
$\TV$-categories and $\TV$-functors (see \cite{CT03}) does not need to be
cartesian closed, as it is the case of the category $\Top$ of topological spaces
and continuous maps, that is $\Cats{(\mU,\two)}$ for $\mU$ the ultrafilter
monad.

Rosick\'{y} showed in \cite{Ros99} that $\Top$ is weakly cartesian closed, and,
consequently, that its exact completion is cartesian closed. Weak cartesian
closedness of $\Top$ follows from the existence of enough injectives in its full
subcategory $\Top_0$ of $T0$-spaces and the fact that they are exponentiable,
and this feature, together with several good properties of $\Top$, gives
cartesian closedness of its exact completion. More precisely, Rosick\'{y} has
shown in \cite{Ros99} the following theorem.

\begin{theorem}\label{th:base}
  Let $\catC$ be a complete, infinitely extensive and well-powered category with
  (reg epi, mono)-factorizations such that $f\times 1$ is an epimorphism
  whenever $f$ is a regular epimorphism. Then the exact completion of $\catC$ is
  cartesian closed provided that $\catC$ is weakly cartesian closed.
\end{theorem}

In this paper we use the setting of $\TV$-categories, for a quantale $\V$ and a
$\Set$-monad $\mT$ laxly extended to $\Rels{\V}$ to conclude, in a unified way,
that several topological categories over $\Set$ share with $\Top$ the cartesian
closedness of the exact completion. This was recently used by Ad\'{a}mek and
Rosick\'{y} in the study of free completions of categories \cite{AR18}. In fact,
the category $\Cats{\TV}$ is topological over $\Set$ \cite{CH03, CT03}, hence
complete and with (reg epi, mono)-factorizations such that $f\times 1$ is an
epimorphism whenever $f$ is, and it is infinitely extensive \cite{MST06}. To
assure weak cartesian closedness of $\Cats{\TV}$ we consider two distinct
scenarios, either restricting to the case that $\V$ is a frame -- so that its
monoidal structure is the cartesian one -- or considering the case that the lax
extension is determined by a $\mT$-algebraic structure on $\V$, as introduced in
\cite{Hof07} under the name of topological theory. In the latter case the proof
generalizes Rosick\'{y}'s proof for $\Top_0$, after observing that, using the
Yoneda embedding of \cite{CH09,Hof11}, every separated $\TV$-category can be
embedded in an injective one, and, moreover, these are exponentiable in
$\Cats{\TV}$. For general $\TV$-categories one proceeds again as in
\cite{Ros99}, using the fact that the reflection of $\Cats{\TV}$ into its full
subcategory of separated $\TV$-categories preserves finite products. As observed
by Rosick\'y, the exact completion of $\Top$ relates to the cartesian closed
category of equilogical spaces \cite{BBS04}. Analogously, our approach leads to
the study of generalized equilogical spaces, as developed in \cite{Rib18}.

The paper is organized as follows. In Section~\ref{sec:categ-tv-categ} we
introduce $\TV$-categories and list their properties used throughout the
paper. In Section~\ref{sec:expon-tv-categ} we revisit the exponentiability
problem in $\Cats{\TV}$, establishing a sufficient criterion for
exponentiability which generalizes the results obtained in
\cite{Hof07,HS15a}. In Section~\ref{sec:repr-tv-categ} we study the properties
of injective $\TV$-categories which will be used in the forthcoming section to
conclude that, under suitable assumptions, injective $\TV$-categories are
exponentiable (Theorem~\ref{d:thm:2}). This result will allow us to conclude, in
Theorem~\ref{th:wcc}, that $\Cats{\TV}$ is weakly cartesian closed, and,
finally, thanks to Theorem \ref{th:base}, that the exact completion of
$\Cats{\TV}$ is cartesian closed. We conclude our paper with a section on
examples, which include, among others, metric spaces, approach spaces,
probabilistic metric spaces, and bitopological spaces.

\section{The category of $\TV$-categories}
\label{sec:categ-tv-categ}

Throughout \emph{$\V$ is a commutative and unital quantale}, i.e $\V$ is a
complete lattice with a symmetric and associative tensor product $\otimes$, with
unit $k$ and right adjoint $\hom$, so that $u\otimes v\leq w$ if, and only if,
$v\leq\hom(u,w)$, for all $u,v,w\in \V$. Further assume that \emph{$\V$ is a
  Heyting algebra}, so that $u\wedge-$ also has a right adjoint, for every
$u\in\V$. We denote by $\Rels{\V}$ the 2-category of \emph{$\V$-relations} (or
$\V$-matrices), having as objects sets, as 1-cells $\V$-relations $r:X\relto Y$,
i.e. maps $r:X\times Y\to \V$, and 2-cells $\varphi:r\to r'$ given by
componentwise order $r(x,y)\leq r'(x,y)$.  Composition of 1-cells is given by
relational composition. $\Rels{\V}$ has an involution, given by transposition:
the transpose of $r:X\relto Y$ is $r^\circ:Y\relto X$ with
$r^\circ(y,x)=r(x,y)$.

We fix a \emph{non-trivial monad $\mT=(T,m,e)$ on $\Set$ satisfying (BC)},
i.e. $T$ preserves weak pullbacks and the naturality squares of the natural
transformation $m$ are weak pullbacks (see \cite{CHJ14}). In general we do not
assume that $T$ preserves products. Later we will make use of the comparison map
$\can_{X,Y}:T(X\times Y)\to TX\times TY$ defined by
$\can_{X,Y}(\fw)=(T\pi_X(\fw),T\pi_Y(\fw))$ for all $\fw\in T(X\times Y)$, where
$\pi_X$ and $\pi_Y$ are the product projections. Moreover, we assume that
\emph{$\mT$ has an extension to $\Rels{\V}$}, which we also denote by $\mT$, in
the following sense:
\begin{itemize}
\item[--] there is a lax functor $T:\Rels{\V}\to\Rels{\V}$ which extends
  $T:\Set\to\Set$;
\item[--] $T(r^\circ)=(Tr)^\circ$ for all $\V$-relations $r$;
\item[--] the natural transformations $e:1_{\Rels{\V}}\to T$ and $m:T^2\to T$
  become op-lax; that is, for every $r:X\relto Y$,
  \begin{align*}
    e_Y\cdot r\le T r\cdot e_X, && m_Y\cdot T T r\le T r\cdot m_X.\\
    \xymatrix{X\ar[r]^-{e_X}\ar[d]|{\object@{|}}_r\ar@{}[dr]|{\le} & T X\ar[d]|{\object@{|}}^{T r}\\ Y\ar[r]_-{e_Y} & T Y} &&
                                                                                                                              \xymatrix{T T X\ar[r]^-{m_X}\ar[d]|{\object@{|}}_{T T r} \ar@{}[dr]|{\le} & T X\ar[d]|{\object@{|}}^{T r}\\ T T Y\ar[r]_-{m_Y} & T Y}
  \end{align*}
\end{itemize}
We note that our conditions are stronger than those used in \cite{HST14}.

A \emph{$\TV$-category} is a pair $(X,a)$ where $X$ is a set and $a:TX\relto X$
is a $\V$-relation such that
\begin{align*}
  \xymatrix{X\ar[r]^{e_X}\ar[dr]_{1_X}^{\le} & TX\ar[d]|-{\object@{|}}^a\\ &X}
                                                                           &&\text{and}&&
                                                                                          \xymatrix{T^2X\ar[r]^{m_X}\ar[d]|-{\object@{|}}_{Ta}\ar@{}[dr]|{\le} &
                                                                                                                                                                 TX\ar[d]|-{\object@{|}}^a\\ TX\ar[r]|-{\object@{|}}_a & X}
\end{align*}
that is, the map $a:TX\times X\to\V$ satisfies the conditions:
\begin{itemize}
\item[(R)] for each $x\in X$, $k\leq a(e_X(x),x)$;
\item[(T)] for each $\fX\in T^2X, \fx\in TX, x\in X$,
  $Ta(\fX,\fx)\otimes a(\fx,x)\leq a(m_X(\fX),x)$.
\end{itemize}
Given $\TV$-categories $(X,a)$, $(Y,b)$, a \emph{$\TV$-functor}
$f:(X,a)\to(Y,b)$ is a map $f:X\to Y$ such that
\[
  \xymatrix{TX\ar[r]^{Tf}\ar[d]|-{\object@{|}}_a\ar@{}[dr]|{\le} &
    TY\ar[d]|-{\object@{|}}^b\\ X\ar[r]_f & Y}
\]
that is, for each $\fx\in TX$ and $x\in X$, $a(\fx,x)\leq b(Tf(\fx),f(x))$; $f$
is said to be \emph{fully faithful} when this inequality is an equality.

$\TV$-categories and $\TV$-functors form the category $\Cats{\TV}$. If
$(X,a:TX\relto X)$ satisfies (R) (and not necessarily (T)), we call it a
\emph{$\TV$-graph}. The category $\Gphs{\TV}$, of $\TV$-graphs and
$\TV$-functors, contains $\Cats{\TV}$ as a full reflective subcategory.

We present the examples in detail in the last section. We mention here, however,
that the leading examples are obtained when one considers the quantale
$\two=(\{0,1\},\le,\&,1)$ and Lawvere's real half-line
$P_+=([0,\infty],\geq,+,0)$, the identity monad $\mId$ and the ultrafilter monad
$\mU$ on $\Set$. Thus we obtain the following examples:
\begin{itemize}
\item[--] $\Cats{(\mId,\V)}$ is the category of $\V$-categories and
  $\V$-functors; in particular, $\Cats{(\mId,\two)}$ is the category $\Ord$ of
  (pre)ordered sets and monotone maps, while $\Cats{(\mId,P_+)}$ is the category
  $\Met$ of Lawvere's metric spaces and non-expansive maps (see \cite{Law73}).
\item[--] $\Cats{(\mU,\two)}$ is the category $\Top$ of topological spaces and
  continuous maps.
\item[--] $\Cats{(\mU,P_+)}$ is the category $\App$ of Lowen's approach spaces
  and non-expansive maps (see \cite{Low97}).
\end{itemize}

We recall (see \cite[Definition~21.1]{AHS90}) that a functor $G:\catA\to\catB$
is said to be \emph{topological} if every source $(f_i:B\to GA_i)_{i\in I}$ in
$\catB$ has a unique $G$-initial lift $(\overline{f_i}:A\to A_i)_{i\in I}$. The
following was proved in \cite{CH03} (see also \cite{CT03}).
\begin{theorem}
  The forgetful functors $\Cats{\TV}\to\Set$ and $\Gphs{\TV}\to\Set$ are
  topological.
\end{theorem}

This shows, in particular, that (see \cite[Chapter 21]{AHS90} for details):
\begin{itemize}
\item[--] $\Cats{\TV}$ is complete and cocomplete.
\item[--] Monomorphisms in $\Cats{\TV}$ are the morphisms whose underlying map
  is injective; therefore, since the $\TV$-structures on any set form a set,
  $\Cats{\TV}$ is well-powered.
\item[--] Every topological category over $\Set$ has two factorization systems,
  (reg epi, mono) and (epi, reg mono); in $\Cats{\TV}$ the former one is in
  general not stable (that is, regular epimorphisms need not be stable under
  pullback -- $\Top$ is such an example), but the latter one is. Indeed,
  epimorphisms in $\Cats{\TV}$ are the $\TV$-functors which are surjective as
  maps, the forgetful functor $\Cats{\TV}\to\Set$ preserves pullbacks, and
  surjective maps are stable under pullback in $\Set$. Therefore, as
  $f\times 1_Z$ is the pullback of $f:X\to Y$ along $\pi_Y:Y\times Z\to Y$, we
  conclude that $f\times 1_Z$ is an epimorphism provided $f$ is.
\end{itemize}

$\Cats{\TV}$ has a natural structure of 2-category: for $\TV$-functors
$f,g:(X,a)\to(Y,b)$, $f\leq g$ if $g\cdot a\leq b\cdot Tf$. This condition can
be equivalently written as $k\leq b(e_Y(f(x)),g(x))$ for every $x\in X$ (see
\cite{CT03} for details). We write $f\simeq g$ if $f\leq g$ and $g\leq f$.

Extensivity of $\Cats{\TV}$ was studied in \cite{MST06}:
\begin{theorem}
  $\Cats{\TV}$ is infinitely extensive.
\end{theorem}

In general $\Cats{\TV}$ is not cartesian closed, while $\Gphs{\TV}$ is. In fact,
the following was proved in \cite{CHT03}:

\begin{theorem}
  $\Gphs{\TV}$ is a quasi-topos.
\end{theorem}

We also note that the tensor product of $\V$ induces a canonical structure $c$
on $X\times Y$ defined by
\[
  c(\fw,(x,y))=a(T\pi_X(\fw),x)\otimes b(T\pi_Y(\fw),y),
\]
where $\fw\in T(X\times Y)$, $x\in X$, $y\in Y$. We put
\[
  (X,a)\otimes(Y,b)=(X\times Y,c),
\]
and this construction is in an obvious way part of a functor
$\otimes:\Gphs{\TV}\times\Gphs{\TV}\to\Gphs{\TV}$. However, the tensor product
of two $\TV$-categories is in general not a $\TV$-category (see
\cite[Lemma~6.1]{Hof07}).

Weak cartesian closedness of $\Cats{\TV}$ needs a thorough study of injective
$\TV$-categories and some extra conditions. This is the subject of the following
sections.

\section{Exponentiable $\TV$-categories}
\label{sec:expon-tv-categ}

Recall that an object $C$ of a category $\catC$ with finite products is
\emph{exponentiable} whenever the functor $C\times-:\catC\to\catC$ has a right
adjoint. The category $\catC$ is \emph{cartesian closed} if every object $C$ of
$\catC$ is exponentiable.  Equivalently, if for each pair of objects $A,B$ of
$\catC$ there exists an object $\langle A,B\rangle$ and a morphism
$\ev:\langle A,B\rangle\times A\to B$ such that, for each morphism
$f:C\times A\to B$ there exists a unique morphism
$\overline{f}:C\to\langle A,B\rangle$ with
$\ev\cdot (\overline{f}\times 1_A)=f$. Dropping uniqueness of $\overline{f}$
gives the notion of \emph{weakly cartesian closed category}.

In this section we present a sufficient condition for a $\TV$-category $X$ to be
exponentiable in $\Cats{\TV}$, which generalises \cite[Theorem~4.3]{Hof06} and
\cite[Theorem~6.5]{Hof07}. To start, we recall that $\Cats{\TV}$ can be fully
embedded into the cartesian closed category $\Gphs{\TV}$.  Here, for
$\TV$-graphs $(X,a)$ and $(Y,b)$, the exponential $\langle (X,a),(Y,b)\rangle$
has as underlying set
\[
  Z:=\{h:(X,a)\times (1,e_1^\circ)\to(Y,b)\mid h\text{ is a $\TV$-functor}\},
\]
which becomes a $\TV$-graph when equipped with the largest structure $b^a$
making the evaluation map
\[
  \ev:Z \times X \to Y,\; (h,x)\mapsto h(x)
\]
a $\TV$-functor: for $\fp\in TZ$ and $h\in Z$, put
\begin{equation*}
  b^a(\fp,h)
  =\bigvee\{v\in\V\mid \forall\fq\in (T\pi_Z)^{-1}(\fp),x\in X\;.\; a(T\pi_X(\fq),x)\wedge v\le b(T\!\ev(\fq),h(x))\},
\end{equation*}
where $\pi_X$ and $\pi_Z$ are the product projections.  Note that the supremum
above is even a maximum since $-\wedge-$ distributes over suprema.

Given $\V$-relations $r:X\relto X'$ and $s:Y\relto Y'$, we define in $\Rels{\V}$
$r\owedge s:X\times Y\relto X'\times Y'$ by
$(r\owedge s)((x,y),(x',y'))=r(x,x')\wedge s(y,y')$. That is,
$r\owedge s=(\pi_{X'}^\circ\cdot r\cdot\pi_X)\wedge(\pi_{Y'}^\circ\cdot
s\cdot\pi_Y)$ in the ordered set $\Rels{\V}(X\times Y,X'\times Y')$.
\begin{theorem}\label{d:thm:1}
  Assume that the diagram
  \begin{equation}\label{infi}
    \xymatrix{T(X\times Y)\ar[d]|-{\object@{|}}_{T(r\owedge s)}\ar[r]^{\can_{X,Y}} & TX\times TY\ar[d]|-{\object@{|}}^{(Tr)\owedge(Ts)}\\
      T(X'\times Y')\ar[r]_{\can_{X',Y'}} & TX'\times TY'}
  \end{equation}
  commutes, for all $\V$-relations $r:X\relto X'$ and $s:Y\relto Y'$. Then a
  $\TV$-category $(X,a)$ is exponentiable provided that
  \begin{align}\label{d:eq:1}
    \bigvee_{\fx\in TX}(Ta(\fX,\fx)\wedge u)\otimes(a(\fx,x)\wedge v)\ge a(m_X(\fX),x)\wedge(u\otimes v),
  \end{align}
  for all $\fX\in TTX$, $x\in X$ and $u,v\in\V$.
\end{theorem}
\begin{proof}
  We show that the $\TV$-graph structure $b^a$ on $Z$ is transitive, for each
  $\TV$-category $(Y,b)$. To this end, let $\fP\in TTZ$, $\fp\in TZ$, $h\in Z$,
  $x\in X$ and $\fw\in T(Z\times X)$ with $T\pi_Z(\fw)=m_{Z}(\fP)$. We have to
  show that
  \[
    (T(b^a)(\fP,\fp)\otimes b^a(\fp,h))\wedge a(T\pi_X(\fw),x) \le
    b(T\ev(\fw),h(x)).
  \]
  Since $m$ has (BC), there is some $\fQ\in TT(Z\times X)$ with
  $TT\pi_Z(\fQ)=\fP$ and $m_{Z\times X}(\fQ)=\fw$. Hence,
  $m_X(TT\pi_X(\fQ))=T\pi_X(\fw)$, and we calculate:
  \begin{multline*}
    (T(b^a)(\fP,\fp)\otimes b^a(\fp,h))\wedge a(T\pi_X(\fw),x)\\
    \begin{aligned}
      &\le \bigvee_{\fx\in TX}((T(b^a)(TT\pi_Z(\fQ),\fp)\wedge
      Ta(TT\pi_X(\fQ),\fx))
      \otimes (b^a(\fp,h)\wedge a(\fx,x))&& \text{(by \eqref{d:eq:1})}\\
      &\le \bigvee_{\fx\in TX} \bigvee_{\fq\in\can^{-1}(\fp,\fx)}
      T(b^a\owedge a)(T\can_{Z,X}(\fQ),\fq)\otimes (b^a\owedge a)(\can_{Z,X}(\fq),(h,x)) && \text{(using \eqref{infi})}\\
      &=\bigvee_{\fq\in (T\pi_Z)^{-1}(\fp)} T(b^a\owedge
      a)(T\can_{Z,X}(\fQ),\fq)\otimes
      (b^a\owedge a)(\can_{Z,X}(\fq),(h,x))\\
      &=\bigvee_{\fq\in (T\pi_Z)^{-1}(\fp)} T(b^a\times a)(\fQ,\fq)\otimes(b^a\times a)(\fq,(h,x))\\
      &\le \bigvee_{\fq\in (T\pi_Z)^{-1}(\fp)} Tb(TT\ev(\fQ),T\ev(\fq))\otimes b(T\ev(\fq),h(x))\\
      &\le b(m_Y\cdot TT\ev(\fQ),h(x))=b(T\ev(\fw),h(x)).
    \end{aligned}
  \end{multline*}
\end{proof}

\begin{remark}
  We note that the inequality
  $\can_{X',Y'}\cdot T(r\owedge s)\le ((Tr)\owedge(Ts))\cdot \can_{X,Y}$ is
  automatically true. Firstly, this inequality is equivalent to
  $T(r\owedge s)\le\can_{X',Y'}^\circ\cdot ((Tr)\owedge(Ts))\cdot \can_{X,Y}$;
  secondly,
  \begin{align*}
    T(r\owedge s)
    &= T((\pi_{X'}^\circ\cdot r\cdot\pi_X)\wedge(\pi_{Y'}^\circ\cdot
      s\cdot\pi_Y))\\
    &\le T(\pi_{X'}^\circ\cdot r\cdot\pi_X)\wedge T(\pi_{Y'}^\circ\cdot
      s\cdot\pi_Y)\\
    &\leq\can_{X',Y'}^\circ\cdot ((Tr)\owedge(Ts))\cdot \can_{X,Y}.
  \end{align*}
\end{remark}

It is worthwhile to notice that, when $V$ is a frame, that is $\otimes=\wedge$,
the condition above is equivalent to
\begin{align*}
  \bigvee_{\fx\in TX}Ta(\fX,\fx)\wedge a(\fx,x)\ge a(m_X(\fX),x),
\end{align*}
for all $\fX\in TTX$ and $x\in X$. Therefore:

\begin{corollary}
  When $V$ is a frame and \eqref{infi} commutes for all $\V$-relations
  $r:X\relto X'$ and $s:Y\relto Y'$, a $\TV$-category $(X,a)$ is exponentiable
  provided that
  \[
    a\cdot m_X=a\cdot Ta.
  \]
\end{corollary}

\section{Injective and representable $\TV$-categories}
\label{sec:repr-tv-categ}

In this section we recall an important class of $\TV$-categories, the so-called
\emph{representable} ones. More information on this type of $\TV$-categories can
be found in \cite{CCH15,HST14}. We also recall from \cite{CH09,Hof07,Hof11} that
every injective $\TV$-category is representable.

Based on the lax extension of the $\Set$-monad $\mT=\monad$ to $\Rels{\V}$,
$\mT$ admits a natural extension to a monad on $\Cats{\V}$, in the sequel also
denoted by $\mT=\monad$ (see \cite{Tho09}). Here the functor
$T:\Cats{\V}\to\Cats{\V}$ sends a $\V$-category $(X,a_0)$ to $(TX,T a_0)$, and
$e_X:X\to TX$ and $m_X:TTX\to TX$ become $\V$-functors for each $\V$-category
$X$. The Eilenberg--Moore algebras for this monad can be described as triples
$(X,a_0,\alpha)$ where $(X,a_0)$ is a $\V$-category and $(X,\alpha)$ is an
algebra for the $\Set$-monad $\mT$ such that $\alpha:T(X,a_0)\to(X,a_0)$ is a
$\V$-functor. For $\mT$-algebras $(X,a_0,\alpha)$ and $(Y,b_0,\beta)$, a map
$f:X\to Y$ is a homomorphism $f:(X,a_0,\alpha)\to(Y,b_0,\beta)$ precisely if $f$
preserves both structures, that is, whenever $f:(X,a_0)\to(Y,b_0)$ is a
$\V$-functor and $f:(X,\alpha)\to(Y,\beta)$ is a $\mT$-homomorphism.

There are canonical adjoint functors
\[
  (\Cats{\V})^\mT\xymatrix@=8ex{\ar@{}[r]|{\top}\ar@<1mm>@/^{2mm}/[r]^{K} &
    \ar@<1mm>@/^{2mm}/[l]^{M}}\Cats{\TV}.
\]
The functor $K$ associates to each $X=(X,a_0,\alpha)$ in $(\Cats{\V})^\mT$ the
$\TV$-category $KX=(X,a)$, where $a=a_0\cdot\alpha$, and keeps morphisms
unchanged. Its left adjoint $M:\Cats{\TV}\to(\Cats{\V})^\mT$ sends a
$\TV$-category $(X,a)$ to $(TX,T a\cdot m_X^\circ,m_X)$ and a $\TV$-functor $f$
to $Tf$. Via the adjunction $M\dashv K$ one obtains a lifting of the
$\Set$-monad $\mT=\monad$ to a monad on $\Cats{\TV}$, also denoted by
$\mT=\monad$.

In this setting we can define `duals' in $(\Cats{\V})^\mT$ and carry them into
$\Cats{\TV}$. Indeed, since $T:\Rels{\V}\to\Rels{\V}$ commutes with the
involution $(-)^\circ$: for every $\mT$-algebra $X=(X,a_0,\alpha)$ also
$(X,a_0^\circ,\alpha)$ is a $\mT$-algebra. Moreover, if $(X,a)$ is a
$\TV$-category, we define $X^\op$ by mapping $(X,a)$ into $(\Cats{\V})^\mT$ via
$M$, dualizing the image in $(\Cats{\V})^\mT$, and then carrying it back to
$\Cats{\TV}$; that is,
\[
  X^\op=K((M(X,a))^\op)=(TX,m_X\cdot (Ta)^\circ\cdot m_X).
\]

Since the monad $\mT=\monad$ on $\Cats{\TV}$ is lax idempotent (i.e, of
Kock-Z\"oberlein type), an algebra structure $\alpha:TX\to X$ on a
$\TV$-category $X$ is left adjoint to the unit $e_X:X\to TX$. We call a
$\TV$-category $X$ \emph{representable} whenever $e_X:X\to TX$ has a left
adjoint in $\Cats{\TV}$; equivalently, whenever there is some $\TV$-functor
$\alpha:TX\to X$ with $\alpha\cdot e_X\simeq 1_X$, since then
\[
  e_X\cdot\alpha=T\alpha\cdot e_{TX} \ge T\alpha\cdot Te_X\simeq 1_{TX}.
\]
However, a left adjoint $\alpha:TX\to X$ to $e_X$ is in general only a
pseudo-algebra structure on $X$, that is,
\begin{align*}
  \alpha\cdot e_X&\simeq 1_X
  &\text{and}&& \alpha\cdot T\alpha&\simeq \alpha\cdot m_X.
\end{align*}
For every representable $\TV$-category $(X,a)$, the structure $a:TX\relto X$ can
be decomposed as $a=a_0\cdot \alpha$, where $a_0=a\cdot e_X$ denotes the
underlying $\V$-category structure.

A $\TV$-category $X$ is \emph{injective} whenever, for each fully faithful
$h:A\to B$ in $\Cats{\TV}$ and each $\TV$-functor $f:A\to X$, there is a
$\TV$-functor $g:B\to X$ with $g\cdot h\simeq f$.

\begin{proposition}
  Every injective $\TV$-category is representable.
\end{proposition}
\begin{proof}
  Let $X$ be an injective $\TV$-category. The $\TV$-functor
  $e_X:(X,a)\to (TX,Ta\cdot m_X^\circ\cdot m_X)$ is an embedding. Indeed, $e_X$
  is injective because the monad $T$ is non-trivial, and it is fully faithful:
  \[
    e_X^\circ\cdot Ta\cdot m_X^\circ\cdot m_X\cdot Te_X\leq a\cdot Ta\cdot
    m_X^\circ\leq a\cdot m_X\cdot m_X^\circ\leq a.
  \]
  Hence, there is a $\TV$-functor $\alpha:TX\to X$ with
  $\alpha\cdot e_X\simeq 1_X$, and so $X$ is representable.
\end{proof}

\section{Injective $\TV$-categories are exponentiable}
\label{sec:inject-tv-categ}

In Section~\ref{sec:catstv-weakly-cart} we will show that, under some
conditions, $\Cats{\TV}$ is weakly cartesian closed. Notably, we will use that
every $\TV$-category can be embedded into an injective one; which, by the main
result of this section, implies that every $\TV$-category can be embedded into
an exponentiable one. We hasten to remark that this is easily seen to be
fulfilled for $\mT$ being the identity monad, witnessed by the \emph{Yoneda
  embedding} (see \cite{Law73})
\begin{equation*}
  \yoneda_X:X\to PX:=\V^{X^\op}.
\end{equation*}
Here $PX$ is the free cocompletion of $X$; being cocomplete, $PX$ is injective.

To treat the general case, we \emph{will consider from now on only} extensions
of the monad $\mT$ to $\Rels{\V}$ given by a $\mT$-algebra structure
$\xi:TV\to V$ on $\V$, so that we are dealing with a \emph{strict topological
  theory} in the sense of \cite{Hof07}. In this case, the extension of
$T:\Set\to\Set$ to $\Rels{\V}$ is defined by
\begin{align*}
  Tr:TX\times TY &\to\V\\
  (\fx,\fy) &\mapsto\bigvee\left\{\xi\cdot Tr(\fw)\;\Bigl\lvert\;\fw\in T(X\times Y), T\pi_X(\fw)=\fx,T\pi_Y(\fw)=\fy\right\}
\end{align*}
for each $\V$-relation $r:X\times Y\to\V$.

In order to obtain a Yoneda embedding, we consider the $\mT$-algebra
$(V,\hom,\xi)$ which is mapped by $K$ into the important $\TV$-category
$(V, \hom_\xi)$, where $\hom_\xi=\hom\cdot\,\xi$ (see
Section~\ref{sec:repr-tv-categ}). The proof of the following result can be found
in \cite{CH09} and \cite{Hof11}.

\begin{theorem}\label{d:thm:3}
  If the extension of $\mT$ to $\Rels{V}$ is induced by a strict topological
  theory, then, for every $\TV$-category $(X,a)$, the $\V$-relation
  $a:TX\relto X$ defines a $\TV$-functor
  \[
    a:X^\op\otimes X\to(\V,\hom_\xi).
  \]
  Moreover, the $\otimes$-exponential mate
  $\yoneda_X=\,^\ulcorner\! a^\urcorner:X\to\V^{X^\op}$ of $a$ is fully
  faithful, and the $\TV$-category $PX=\V^{X^\op}$ is injective.

  In fact, this construction defines a functor $P:\Cats{\TV}\to\Cats{\TV}$ and
  $\yoneda=(\yoneda_X)_X$ is a natural transformation
  $\yoneda:1_{\Cats{\TV}}\to P$.
\end{theorem}

Since $\yoneda_X$ is fully faithful, when $X$ is injective there exists a
$\TV$-functor $\Sup_X:PX\to X$ such that $\Sup_X\cdot\yoneda_X\simeq 1_X$. As
shown in \cite[Theorem 2.7]{Hof11}, $\Sup_X\dashv\yoneda_X$. Moreover, for each
$\TV$-category $(X,a)$, $\yoneda_X$ is one-to-one if, and only if, $(X,a)$ is
\emph{separated}, i.e.\ for every $f,g:(Y,b)\to(X,a)$, $f\simeq g$ only if $f=g$
(see \cite{HT10}, for example). It follows immediately that, for an injective
$\TV$-functor $f:X\to Y$ where $Y$ is separated, also $X$ is.

\begin{lemma}
  \label{d:lem:2}
  The $\otimes$-exponential $Y^X$ is separated, for every separated
  $\TV$-category $Y$ and every representable $\TV$-category $X$; in particular,
  $PX$ is separated, for every $\TV$-category $X$.
\end{lemma}
\begin{proof}
  See \cite[Corollary~4.12~(2)]{HT10}.
\end{proof}

\begin{corollary}
  Every separated $\TV$-category embeds into an injective $\TV$-category.
\end{corollary}

In Section~\ref{sec:categ-tv-categ} we introduced the tensor product
$X\otimes Y$ of $\TV$-graphs $X$ and $Y$. We remark that, in the setting of a
strict topological theory, $X\otimes Y$ is a $\TV$-category provided that $X$
and $Y$ are so (see \cite{Hof07}).

The result promised in the title of this section was shown in
\cite[Proposition~2.7]{Hof14} for the special case of $\otimes=\wedge$:

\begin{proposition}
  If the quantale $\V$ is a frame and \eqref{infi} commutes for all
  $\V$-relations $r:X\relto X'$ and $s:Y\relto Y'$, then every representable
  $\TV$-category is exponentiable. In particular, in this case every injective
  $\TV$-category is exponentiable.
\end{proposition}

To treat the general case, we will make use of the following conditions:
\begin{assumptions}
  \label{d:ass:1}
  \begin{enumerate}
  \item\label{d:item:8} The diagram \eqref{infi} commutes, for all
    $\V$-relations $r:X\relto X'$ and $s:Y\relto Y'$.
  \item\label{inj} For all $u,v,w\in \V$,
    \begin{equation*}
      w\wedge (u\otimes v)=\bigvee\{u'\otimes v'\;|\;u'\leq u, v'\leq v,u'\otimes v'\leq w\};
    \end{equation*}
    or, equivalently, every injective $\V$-category is exponentiable: see
    \cite[Theorem~5.3]{HR13}.
  \item\label{tensor} For every $\V$-relation $a:X\relto Y$ and $u\in \V$,
    \begin{equation*}
      T(a\otimes u)=Ta\otimes u,
    \end{equation*}
    where $a\otimes u$ is the $\V$-relation defined by
    $(a\otimes u)(x,y)=a(x,y)\otimes u$.
  \item\label{fctors} The maps $\V\otimes \V\xrightarrow{\quad\otimes\quad}\V$
    and $X\xrightarrow{(-,u)}X\otimes \V$ are $\TV$-functors, for all $u\in V$.
  \end{enumerate}
\end{assumptions}

These morphisms induce an interesting action of $\V$ on every injective
$\TV$-category $(X,a)$ as follows.  The $\TV$-functor
\[
  \xymatrix{X^\op\otimes X\otimes \V\ar[rr]^-{a\otimes 1}&&\V\otimes
    \V\ar[rr]^-{\otimes}&&\V}
\]
induces a $\TV$-functor $\tilde{a}:X\otimes \V\to PX$. We denote the composite
\[
  \xymatrix{X\otimes \V\ar[rr]^-{\tilde{a}}&&PX\ar[rr]^-{\Sup_X}&&X}
\]
by $\oplus$, and
\[
  \xymatrix{X\ar[rr]^-{(-,u)} && X\otimes \V\ar[rr]^-{\tilde{a}} &&
    PX\ar[rr]^-{\Sup_X}&&X,}
\]
assigning to each $x\in X$ an element $x\oplus u$ in $X$, by $-\oplus u$.

Analogously we will write $\fx\oplus u$ for $T(-\oplus u)(\fx)$, for every
$\fx\in TX$ and $u\in \V$. Note that $\TV$-functoriality of $-\oplus u$ can be
written as
\[
  a(\fx,y)\leq a(\fx\oplus u,y\oplus u),
\]
for every $\fx\in TX$ and $y\in X$.

\begin{lemma}\label{calculus}
  Assuming \ref{d:ass:1}~\eqref{fctors}, for an injective $\TV$-category
  $(X,a)$, with $a=a_0\cdot\alpha$ as usual, the following holds, for every
  $x,y\in X$, $\fx\in TX$ and $u\in \V$:
  \begin{enumerate}
  \item\label{d:item:1} $a_0(x\oplus u,y)=\hom(u,a_0(x,y))$;
  \item\label{d:item:2} $a_0(x,y\oplus u)\geq a_0(x,y)\otimes u$;
  \item\label{d:item:3} $a(\fx\oplus u,y)\geq \hom(u,a(\fx,y))$;
  \item\label{d:item:4} $a(\fx,y\oplus u)\geq a(\fx,y)\otimes u$.
  \end{enumerate}
  Moreover, if, in addition, \ref{d:ass:1}~\eqref{tensor} holds, then, for every
  $\fX\in T^2X$, $\fy\in TX$, $u\in \V$,
  \begin{enumerate}[resume]
  \item\label{d:item:5} $Ta(\fX,\fy\oplus u)\geq Ta(\fX,\fy)\otimes u$.
  \end{enumerate}
\end{lemma}

\begin{proof}
  (\ref{d:item:1}) For every $x,y\in X$ and $u\in \V$,
  \begin{align*}
    a_0(x\oplus u,y)&=a_0(\Sup_X(\tilde{a}(x,u)),y)&&\text{(by definition of $\oplus$)}\\
                    &=[\tilde{a}(x,u),\yoneda_X(y)]&&\text{(because $\Sup_X\dashv \yoneda_X$)}\\
                    &=\bigwedge_{\fx\in TX}\hom(\tilde{a}(x,u)(\fx),\yoneda_X(y)(\fx))&&\text{(by definition of $[\; ,\;]$)}\\
                    &=\bigwedge_{\fx\in TX}\hom(a(\fx,x)\otimes u,a(\fx,y))&& \text{(by definition of $\tilde{a}$ and $\yoneda_X(y)$)}\\
                    &=\hom(u,a_0(x,y)),
  \end{align*}
  because, using the fact that $a=a_0\cdot\alpha$ and
  \[
    a_0(\alpha(\fx),x)\otimes u\otimes\hom(u,a_0(x,y))\leq
    a_0(\alpha(\fx),x)\otimes a_0(x,y)\leq a_0(\alpha(\fx),y),
  \]
  for $\fx\in TX$, we can conclude that
  \[
    \hom(u,a_0(x,y))\leq\bigwedge_{\fx\in TX}\hom(a_0(\alpha(\fx),x)\otimes
    u,a_0(\alpha(\fx),y)).
  \]
  Taking $\fx=e_X(x)$, we see that this inequality is in fact an equality as
  claimed.

  (\ref{d:item:2}) Since, by hypothesis, $-\oplus u$ is a $\TV$-functor, and so,
  in particular, a $\V$-functor $(X,a_0)\to (X,a_0)$,
  \[
    a_0(x,y)\leq a_0(x\oplus u,y\oplus u)=\hom(u,a_0(x,y\oplus u)),
  \]
  and then
  \[
    a_0(x,y)\otimes u\leq \hom(u,a_0(x,y\oplus u))\otimes u\leq a_0(x,y\oplus
    u).
  \]

  (\ref{d:item:3}) One has
  \[
    \begin{array}{rcl}
      k&\leq&a_0(\alpha(\fx),\alpha(\fx))=a(\fx,\alpha(\fx))\\
       &\leq&a(\fx\oplus u,\alpha(\fx)\oplus u)\\
       &=&a_0(\alpha(\fx\oplus u),\alpha(\fx)\oplus u).
    \end{array}
  \]
  Using (\ref{d:item:1}) we conclude that
  \[
    \begin{array}{rcl}
      \hom(u,a(\fx,y))&=&a_0(\alpha(\fx)\oplus u,y)\\
                      &\leq&a_0(\alpha(\fx\oplus u),\alpha(\fx)\oplus u)\otimes a_0(\alpha(\fx)\oplus u,y)\\
                      &\leq&a_0(\alpha(\fx\oplus u),y)=a(\fx\oplus u,y).
    \end{array}
  \]

  (\ref{d:item:4}) follows directly from (\ref{d:item:2}), while
  (\ref{d:item:5}) follows from (\ref{d:item:4}).
\end{proof}

\begin{lemma}
  \label{d:lem:1}
  Let $\varphi:\V\to \W$ be a surjective quantale homomorphism; that is,
  $\varphi$ preserves the tensor, the neutral element, and suprema. Then, if
  $\V$ satisfies condition \ref{d:ass:1}~\eqref{inj}, so does $\W$.
\end{lemma}

\begin{theorem}
  \label{d:thm:2}
  Under Assumptions~\ref{d:ass:1}, every injective $\TV$-category is
  exponentiable in $\Cats{\TV}$.
\end{theorem}

\begin{proof}
  Let $\fX\in T^2X$, $x\in X$ and $u,v\in \V$. In order to conclude that
  \[
    \bigvee_{\fx\in TX}(Ta(\fX,\fx)\wedge u)\otimes(a(\fx,x)\wedge v)\geq
    a(m_X(\fX),x)\wedge(u\otimes v),
  \]
  we make use of Hypothesis~\ref{d:ass:1}~\eqref{inj}. Let $u',v'\in\V$ with
  $u'\le u$, $v'\le v$ and $u'\otimes v'\le a(m_X(\fX),x)$. First we note that
  \begin{align*}
    Ta(\fX,T\alpha(\fX)\oplus u')\wedge u
    &\geq (Ta(\fX,T\alpha(\fX))\otimes u')\wedge u & \text{(by \ref{calculus}~(\ref{d:item:5}))}\\
    &=(Ta_0(T\alpha(\fX),T\alpha(\fX))\otimes u')\wedge u\\
    &\geq(k\otimes u')\wedge u=u',
  \end{align*}
  and
  \begin{align*}
    a(T\alpha(\fX)\oplus u',x)
    &\geq\hom(u',a(T\alpha(\fX),x)) & \text{(by \ref{calculus}~(\ref{d:item:3}))}\\
    &=\hom(u',a_0(\alpha(T\alpha(\fX)),x))\\
    &=\hom(u',a_0(\alpha(m_X(\fX)),x))\\
    &=\hom(u',a(m_X(\fX),x)).
  \end{align*}
  Now, from $u'\otimes v'\leq a(m_X(\fX),x)$ and $v'\le v$ we get
  \[
    v'\leq\hom(u',a(m_X(\fX),x))\wedge v\le a(T\alpha(\fX)\oplus u',x)\wedge v,
  \]
  hence
  \[
    u'\otimes v'\leq (Ta(\fX,T\alpha(\fX)\oplus u')\wedge
    u)\otimes(a(T\alpha(\fX)\oplus u',x)\wedge v).
  \]
  Therefore
  $\displaystyle{a(m_X(\fX),x)\wedge(u\otimes v)\le \bigvee_{\fx\in
      TX}(Ta(\fX,\fx)\wedge u)\otimes(a(\fx,x)\wedge v)}$.
\end{proof}

\begin{remark}
  Under Assumptions~\ref{d:ass:1}, it follows from Lemma~\ref{d:lem:2} that the
  exponential $\langle (X,a),(Y,b)\rangle$ is separated, for all separated
  injective $\TV$-categories $(X,a)$ and $(Y,b)$. In fact, with
  $a=a_0\cdot\alpha$, the epimorphism $(X,\alpha)\to(X,a)$ in $\Cats{\TV}$ is
  mapped to the monomorphism
  \begin{equation*}
    \langle (X,a), (Y,b)\rangle
    \quad\xrightarrow{\qquad}\quad \langle (X,\alpha),(Y,b)\rangle=(Y,b)^{(X,\alpha)},
  \end{equation*}
  which proves that $ \langle (X,a), (Y,b)\rangle$ is separated.
\end{remark}

\section{$\Cats{\TV}$ is weakly cartesian closed}
\label{sec:catstv-weakly-cart}

Building on the results of the previous section, in this section we show that,
under some conditions, $\Cats{\TV}$ is weakly cartesian closed. We start by
proving this property for the full subcategory $\Cats{\TV}_\sep$ of $\Cats{\TV}$
of separated $\TV$-categories.

\begin{theorem}
  Under Assumptions~\ref{d:ass:1}, $\Cats{\TV}_\sep$ is weakly cartesian closed.
\end{theorem}

\begin{proof}
  For $X,Y$ separated $\TV$-categories, consider the Yoneda embeddings
  $\yoneda_X:X\to PX$ and $\yoneda_Y:Y\to PY$, and the exponential
  $\langle PX,PY\rangle$. The elements of its underlying set can be identified
  with $\TV$-functors $E\times PX\to PY$ (where $E=(1,e_1^\circ)$ is the
  generator of $\Cats{\TV}$), and the universal morphism
  $\ev:\langle PX,PY\rangle\times PX\to PY$ with the evaluation map:
  $\ev(\varphi,\fx)=\varphi(\fx)$ (where, for simplicity, we identify the set
  $E\times PX$ with $PX$). We can therefore consider
  \[
    \ll X,Y\gg=\{\varphi:E\times PX\to PY\mid
    \varphi(\yoneda_X(X))\subseteq\yoneda_Y(Y)\},
  \]
  with the initial structure with respect to the inclusion
  $\iota:\ll X,Y\gg\to \langle PX,PY\rangle$. Moreover, the morphism
  \[
    \xymatrix{\ll X,Y\gg\times X\ar[rr]^-{\iota\times\yoneda_X} && \langle
      PX,PY\rangle\times PX\ar[rr]^-{\ev} && PY}
  \]
  factors through $\yoneda_Y$ via a morphism
  \[
    \xymatrix{\ll X,Y\gg\times X\ar[rr]^-{\widetilde{\ev}}&&Y.}
  \]
  Next we show that this is a weak exponential in $\Cats{\TV}_\sep$.

  Given any separated $\TV$-category $Z$, and a $\TV$-functor
  $f:Z\times X\to Y$, by injectivity of $PY$ there exists a $\TV$-functor
  $f':Z\times PX\to PY$ making the square below commute. Then, by universality
  of the evaluation map $\ev$, there exists a unique $\TV$-functor
  $\overline{f}:Z\to \langle PX,PY\rangle$ making the bottom triangle commute.
  \[
    \xymatrix{Z\times X\ar[r]^f\ar[d]_{1_Z\times\yoneda_X}&Y\ar[d]^{\yoneda_Y}\\
      Z\times PX\ar[r]^{f'}\ar[d]_{\overline{f}\times 1_{PX}}&PY\\
      \langle PX,PY\rangle\times PX\ar[ru]_{\ev}&}
  \]
  The map $\overline{f}:Z\to \langle PX,PY\rangle$, assigning to each $z\in Z$ a
  map $\overline{f}(z):PX\to PY$, is such that
  $\overline{f}(z)(\yoneda_X(x))
  =\ev(\overline{f}(z),\yoneda_X(x))=\yoneda_Y(f(z,x))$; that is,
  $\overline{f}(z)(\yoneda_X(X))\subseteq\yoneda_Y(Y)$, and this means that
  $\overline{f}(z)\in\ll X,Y\gg$. Hence we can consider the corestriction
  $\tilde{f}$ of $\overline{f}$ to $\ll X,Y\gg$, which is again a $\TV$-functor
  since $\ll X,Y\gg$ has the initial structure with respect to
  $\langle PX,PY\rangle$, so that the following diagram commutes.
  \[
    \xymatrix{\ll X,Y\gg\times X\ar[rr]^-{\widetilde{\ev}}&&Y\\
      Z\times X\ar[rru]_f\ar[u]^{\tilde{f}\times 1_X}&}
  \]
\end{proof}

In order to show that $\Cats{\TV}$ is weakly cartesian closed, we follow the
proof of \cite{Ros99}. Hence, first we show that:
\begin{proposition}
  The reflector $R:\Cats{\TV}\to\Cats{\TV}_\sep$ preserves finite products.
\end{proposition}

\begin{proof}
  We recall that, for any $\TV$-category $(X,a)$,
  $R(X,a)=(\tilde{X},\tilde{a})$, with $\tilde{X}=X/\sim$, where $x\sim y$ if
  $k\leq a(e_X(x),y)\wedge a(e_X(y), x)$, and
  $\tilde{a}=\eta_X\cdot a\cdot(T\eta_X)^\circ$, with $\eta_X:X\to\tilde{X}$ the
  projection. This structure makes $\eta_X$ both an initial and a final morphism
  (see \cite{HST14} for details).

  Let $f:R(X\times Y)\to RX\times RY$ be the unique morphism such that
  $f\cdot\eta_{X\times Y}=\eta_X\times\eta_Y$.
  \[
    \xymatrix{(X\times Y,c)\ar[rrd]_{\eta_X\times\eta_Y}\ar[rr]^-{\eta_{X\times Y}}&&(R(X\times Y),\tilde{c})\ar[d]^f\\
      &&(RX\times RY,d)}
  \]
  From $c(e_{X\times Y}(x,y),(x',y'))=a(e_X(x),x')\wedge b(e_Y(y),y')$ it is
  immediate that $(x,y)\sim (x',y')$ in $X\times Y$ if, and only if, $x\sim x'$
  in $X$ and $y\sim y'$ in $Y$. Therefore, $f$ is a bijection. Assuming the
  Axiom of Choice, so that $T$ preserves surjections, we have, for every
  $\fz\in T(R(X\times Y))$, $(x,y)\in X\times Y$,
  \[
    \begin{array}{rcll}
      \tilde{c}(\fz,[(x,y)])&=&c(\fw,(x,y))&\text{(for any $\fw\in(T\eta_{X\times Y})^{-1}(\fz)$)}\\
                            &=&d(T(\eta_X\times\eta_Y)(\fw),([x],[y]))&\text{(because $\eta_X\times\eta_Y$ is initial)}\\
                            &=&d(Tf(\fz),([x],[y]);
    \end{array}
  \]
  that is, $f$ is initial and therefore an isomorphism.
\end{proof}

\begin{theorem}\label{th:wcc}
  Under Assumptions~\ref{d:ass:1}, $\Cats{\TV}$ is weakly cartesian closed.
\end{theorem}

\begin{proof}
  Given $\TV$-categories $(X,a)$, $(Y,b)$, to build the weak exponential
  $\ll X,Y\gg$ we will show the \emph{cosolution set condition} for the functor
  $-\times(X,a)$.

  For each $\TV$-functor $f:(Z,c)\times(X,a)\to (Y,b)$ we consider its
  reflection $Rf:RZ\times RX\cong R(Z\times X)\to RY$ and we factorise it
  through the weak evaluation in $\Cats{\TV}_\sep$,
  $Rf=\widetilde{\ev}\cdot(\overline{Rf}\times 1_{RX})$, so that in the diagram
  below the outer rectangle commutes.

  Then we define $Z_f=Z/\sim$ by
  \[
    z\sim z'\text{ if both }f(z,x)=f(z',x),\text{ for every }x\in X,\text{ and
    }\overline{Rf}(\eta_Z(z))=\overline{Rf}(\eta_Z(z')),
  \]
  and equip it with the final structure for the projection $q_f:Z\to Z_f$. Then
  $h_f:Z_f\to\ll RX,RY\gg$, with $h_f([z])=\overline{Rf}(\eta_Z(z))$, is a
  $\TV$-functor since its composition with $q_f$ is $\overline{Rf}\cdot\eta_Z$
  and $q_f$ is final. Then we factorise $f$ via the surjection
  $q_f\times 1_X:Z\times X\to Z_f\times X$ as in the diagram below.  Moreover,
  the map $\hat{f}:Z_f\times X\to Y$, with $\hat{f}([z],x)=f(z,x)$, is a
  $\TV$-functor because
  $\eta_Y\cdot\hat{f}=\widetilde{\ev}\cdot(h_f\times\eta_X)$ is and $\eta_Y$ is
  initial.
  \[
    \xymatrix@C=6pt{Z\times X\ar[rrrrrr]^f\ar[dd]_{\eta_Z\times 1_X}\ar[ddrr]^{q_f\times 1_X}&&&&&&Y\ar[dddd]^{\eta_Y}\\
      &&&&&&\\
      RZ\times X\ar[dd]_{\overline{Rf}\times 1_X}&&Z_f\times X\ar[rrrruu]^{\hat{f}}\ar[rr]\ar[ddll]^{h_f\times 1_X}&&(\coprod_g Z_g\times X)\cong(\coprod_gZ_g)\times X\ar[rruu]^{\ev}\\
      &&&&&&\\
      \ll RX,RY\gg\times X\ar[rrr]_{1\times\eta_X}&&&\ll RX,RY\gg\times
      RX\ar[rrr]_{\widetilde{\ev}}&&&RY}
  \]
  Since the cardinality of $Z_f$ is bounded by the cardinality of the set
  $|\ll RX,RY\gg|\times|Y|^{|X|}$, as witnessed by the injective map
  \[
    \begin{array}{rcl}
      Z_f&\to&|\ll RX,RY\gg|\times|Y|^{|X|},\\
      \,[z]&\mapsto&(\overline{Rf}(\eta_Z(z)),f(z,-))
    \end{array}
  \]
  there is only a set of possible $\TV$-categories $Z_f$. Hence we can form its
  coproduct, as in the diagram above, and consider the induced $\TV$-functor
  $\ev:(\coprod_gZ_g)\times X\cong\coprod_g(Z_g\times X)\to Y$ (note that the
  isomorphism follows from extensivity of $\Cats{\TV}$).
\end{proof}

\section{Examples}
\label{sec:examples}

In this section we use Theorem~\ref{th:wcc} to present examples of weakly
cartesian closed categories. Hence, in conjunction with the following theorem
established in \cite{Ros99}, we obtain examples of categories with cartesian
closed exact completion since all other conditions of that theorem are trivially
satisfied in these examples.

\begin{theorem}
  Let $\catC$ be a complete, infinitely extensive and well-powered category in
  which every morphism factorizes as a regular epi followed by a mono, and where
  $f\times 1$ is an epimorphism for every regular epimorphism $f:A\to B$ in
  $\catC$. Then, if $\catC$ is weakly cartesian closed, the exact completion
  $\catC_{\mathrm{ex}}$ of $\catC$ is cartesian closed.
\end{theorem}

We note that, in order to conclude that $\Cats{\TV}$ is weakly cartesian closed,
we have to check whether $\V$ and $\mT$ satisfy Assumptions~\ref{d:ass:1}.

First we analyse examples where $\mT$ is the identity monad. In this particular
setting we only have to check that \ref{d:ass:1}~\eqref{inj} holds. The category
$\Cats{\V}$ is always monoidal closed, as shown in \cite{Law73}. Therefore, when
$\V$ is a frame considered as a quantale, then $\Cats{\V}$ is cartesian
closed. This is the case of $\two$, and so one concludes that \emph{$\Ord$ is
  cartesian closed}. Moreover, for $\V$ the lattice $([0,\infty],\geq)$ with
$\otimes=\wedge$, $\Cats{\V}$ is \emph{the category of ultrametric spaces},
which \emph{is} therefore also \emph{cartesian closed}.

When $\V=P_+$, $\Cats{\V}$ is the category $\Met$ of Lawvere's metric spaces
\cite{Law73}, which is not cartesian closed (see \cite{CH06} for
details). However, the quantale $P_+$ satisfies \ref{d:ass:1}~\eqref{inj}, and
so from Theorem~\ref{th:wcc} it follows that \emph{$\Met$ is weakly cartesian
  closed}.

Metric and ultrametric spaces can be also viewed as categories enriched in a
quantale based on the complete lattice $[0,1]$ with the usual ``less or equal''
relation $\le$, which is isomorphic to $[0,\infty]$ via the map
$[0,1]\to[0,\infty],\,u\mapsto -\ln(u)$ where $-\ln(0)=\infty$. More in detail,
we consider the following quantale operations on $[0,1]$ with neutral element
$1$.
\begin{enumerate}
\item For $\otimes=*$ being the ordinary multiplication, via the isomorphism
  $[0,1]\simeq[0,\infty]$, this quantale is isomorphic to the quantale $P_+$,
  hence $\Cats{[0,1]}\simeq\Met$.
\item For the tensor $\otimes=\wedge$ being infimum, the isomorphism
  $[0,1]\simeq[0,\infty]$ establishes an equivalence between $\Cats{[0,1]}$ and
  the category of ultrametric spaces and non-expansive maps.
\item Another interesting multiplication on $[0,1]$ is the \emph{\L{}ukasiewicz
    tensor} $\otimes=\odot$ given by $u\odot v=\max(0,u+v-1)$. Via the lattice
  isomorphism $[0,1]\to[0,1],\,u\mapsto 1-u$, this quantale is isomorphic to the
  quantale $[0,1]$ with ``greater or equal'' relation $\ge$ and tensor
  $u\otimes v=\min(1,u+v)$ truncated addition. Therefore $\Cats{[0,1]}$ is
  equivalent to the \emph{category of bounded-by-$1$ metric spaces and
    non-expansive maps}. Moreover, with respect to the ``greater or equal''
  relation and truncated addition on $[0,1]$, the map
  \begin{align*}
    [0,\infty]\to [0,1],\,u\mapsto \min(1,u)
  \end{align*}
  is a surjective quantale morphism; therefore, by Lemma~\ref{d:lem:1}, also
  $[0,1]$ with the \L{}ukasiewicz tensor satisfies \ref{d:ass:1}~\eqref{inj}.
\item More generally, every continuous quantale structure $\otimes$ on the
  lattice $[0,1]$ (with Euclidean topology and the usual ``less or equal''
  relation) with neutral element $1$ satisfies \ref{d:ass:1}~\eqref{inj}. This
  can be shown using the fact, proven in \cite{Fau55} and \cite{MS57}, that
  every such tensor $\otimes:[0,1]\times[0,1]\to[0,1]$ is a combination of the
  three operations on $[0,1]$ described above. More precise:
  \begin{enumerate}
  \item For $u,v\in[0,1]$ and $e\in [0,1]$ idempotent with $u\le e\le v$:
    $u\otimes v=\min(u,v)=u$.
  \item For every non-idempotent $u\in[0,1]$, there exist idempotents $e$ and
    $f$ with $e< u< f$ and such that the interval $[e,f]$ (with the restriction
    of the tensor on $[0,1]$ and with neutral element $f$) is isomorphic to
    $[0,1]$ either with multiplication or \L{}ukasiewicz tensor.
  \end{enumerate}
  Now let $w,u,v\in [0,1]$. We may assume $u\le v$. If $u\otimes v\le w$, then
  clearly
  \[
    w\wedge (u\otimes v)=u\otimes v=\bigvee\{u'\otimes v'\;|\;u'\leq u, v'\leq
    v,u'\otimes v'\leq w\}.
  \]
  We consider now $w< u\otimes v\le u\le v$. If $w$ is idempotent, then
  \[
    w=w\otimes v,\quad w\le u,\quad v\le v;
  \]
  otherwise there are idempotents $e$ and $f$ with $e<w<f$ and $[e,f]$ is
  isomorphic to $[0,1]$ either with multiplication or \L{}ukasiewicz tensor.
  \begin{description}
  \item[Case 1] $v\le f$. Then \ref{d:ass:1}~\eqref{inj} holds since
    $w,u\otimes v,u,v\in [e,f]$.
  \item[Case 2] $f < v$. Then $w=w\wedge v =w\otimes v$, $w\le u$ and $v\le v$.
  \end{description}
  We conclude that \emph{$\Cats{[0,1]}$ is weakly cartesian closed}, for every
  continuous quantale structure $\otimes$ on the lattice $[0,1]$ with neutral
  element $1$.
\end{enumerate}

Now let $\V=\Delta$ be the \emph{quantale of distribution functions} (see
\cite{HR13, CH17} for details). As observed in \cite{HR13}, it verifies
\ref{d:ass:1}~\eqref{inj}, and so we can conclude from Theorem~\ref{th:wcc} that
\emph{the category $\Cats{\Delta}$ of probabilistic metric spaces and
  non-expansive maps is weakly cartesian closed.}

When $\mT$ is not the identity monad, some further work is need to guarantee
Assumptions~\ref{d:ass:1}.

\begin{theorem}
  \begin{enumerate}
  \item The tensor product on the quantale $\V$ defines a $\TV$-functor
    $\otimes:\V\otimes\V\to\V$.
  \item\label{d:item:6} Let $u\in V$ satisfying $u\cdot !\ge\xi\cdot Tu$.
    \[
      \xymatrix{T1\ar[r]^{Tu}\ar[d]_{!}\ar@{}[dr]|{\ge} & T\V\ar[d]^{\xi} \\
        1\ar[r]_{u} & \V}
    \]
    Then $(-,u):X\to X\times\V$ is a $\TV$-functor, for every $\TV$-category
    $X$.
  \item\label{d:item:7} Let $u\in V$ satisfying $u\cdot !=\xi\cdot Tu$. Then
    $T(r\otimes u)=(Tr)\otimes u$, for every $\V$-relation $r:X\relto Y$.
  \end{enumerate}
\end{theorem}
\begin{proof}
  The first assertion is \cite[Proposition~1.4(1)]{Hof11}.  To see
  (\ref{d:item:6}), assume that $u\in V$ with $u\cdot !\ge\xi\cdot Tu$. Let
  $(X,a)$ be a $\TV$-category, $\fx\in TX$ and $x\in X$. Considering the map
  $X\xrightarrow{!}1\xrightarrow{u}\V$, we have to show that
  \[
    a(\fx,x)\le a(\fx,x)\otimes\hom(T(u\cdot\, !)(\fx),u),
  \]
  which follows immediately from $u\cdot !\ge\xi\cdot Tu$. Finally, to prove
  (\ref{d:item:7}), let $r:X\relto Y$ be a $\V$-relation and $u\in V$ with
  $u\cdot !=\xi\cdot Tu$. Note that the $\V$-relation $r\otimes u:X\relto Y$ is
  given by
  \[
    X\times Y\xrightarrow{\quad r\quad}\V\xrightarrow{\quad\langle 1_\V,u\cdot
      !\rangle\quad}\V\times\V\xrightarrow{\quad\otimes\quad}\V.
  \]
  Hence, applying the $\Set$-functor $T$ to the functions $r:X\times Y\to\V$ and
  $r\otimes u:X\times Y\to\V$ , we obtain
  \begin{align*}
    \xi\cdot T(r\otimes u)
    &=\xi\cdot T(\otimes)\cdot T\langle 1_\V,u\cdot
      !\rangle\cdot Tr\\
    &=\otimes\cdot(\xi\times\xi)\cdot\can_{X,Y} \cdot T\langle 1_\V,u\cdot
      !\rangle\cdot Tr\\
    &=\otimes\cdot \langle \xi,u\cdot !\cdot \xi\rangle\cdot Tr\\
    &=\otimes\cdot \langle 1_\V,u\cdot !\rangle\cdot\xi\cdot Tr.
  \end{align*}
  Therefore, returning to $\V$-relations, we conclude that
  $T(r\otimes u)=(Tr)\otimes u$.
\end{proof}

\begin{remark}
  If $T1=1$, then $u\cdot !=\xi\cdot Tu$ for every $u\in\V$.
\end{remark}
In order to guarantee Assumptions~\ref{d:ass:1}~\eqref{d:item:8}, we need an
extra condition on $\xi$.

\begin{proposition}
  Assume that
  \[
    \xymatrix{T(\V\times\V)\ar[rr]^-{T(\wedge)}\ar[d]_{\langle\xi\cdot
        T\pi_1,\xi\cdot T\pi_2\rangle}\ar@{}[drr]|{\le} && T\V\ar[d]^\xi\\
      \V\times\V\ar[rr]_-{\wedge} && \V.}
  \]
  Then, for all $\V$-relations $r:X\relto X'$ and $s:Y\relto Y'$,
  \[
    \xymatrix{T(X\times Y)\ar[d]|-{\object@{|}}_{T(r\owedge s)}\ar[r]^{\can_{X,Y}}\ar@{}[dr]|{\ge} & TX\times TY\ar[d]|-{\object@{|}}^{Tr\owedge Ts}\\
      T(X'\times Y')\ar[r]_{\can_{X',Y'}} & TX'\times TY'.}
  \]
\end{proposition}
\begin{proof}
  First we note that, from the preservation of weak pullbacks by $T$, it follows
  that the commutative diagram
  \[
    \xymatrix@+1ex{T(A\times B)\ar[r]^{T(f\times g)}\ar[d]_{\can_{A,B}} & T(X\times Y)\ar[d]^{\can_{X,Y}}\\
      TA\times TB\ar[r]_{Tf\times Tg} & TX\times TY}
  \]
  is also a weak pullback.

  Let $\fw\in T(X\times Y)$, $\fx'\in TX'$ and $\fy'\in TY'$. Put
  $(\fx,\fy)=\can_{X,Y}(\fw)$. By the definition of the extension of $T$ and
  since $\V$ is a Heyting algebra,
  \begin{align*}
    Tr(\fx,\fx')\wedge Ts(\fy,\fy')
    &=\bigvee\left\{\xi\cdot Tr(\fw_1)\wedge\xi\cdot Ts(\fw_2)\;\Bigl\lvert\;
      \begin{aligned}
        \fw_1\in T(X\times X'):\fw_1\mapsto\fx, \fw_1\mapsto\fx'\\
        \fw_2\in T(Y\times Y'):\fw_2\mapsto\fy,\fw_2\mapsto\fy'
      \end{aligned}\right\}.
  \end{align*}
  Note that in
  \[
    \xymatrix{ & T(X\times Y\times X'\times Y')\ar[d]_{\cong}\\
      T(X\times Y)\ar[d]_{\can} & T(X\times X'\times Y\times Y')\ar[l]_-{T(\pi_X\times\pi_Y)}\ar[r]^-{T(r\times s)}\ar[d]_{\can} & T(\V\times\V)\ar[d]_{\can}\ar[r]^-{T(\wedge)}\ar@{}[ddr]|(0.7){\le} & T\V\ar[dd]^\xi\\
      TX\times TY & T(X\times X')\times T(Y\times Y')\ar[l]^-{T\pi_X\times T\pi_Y}\ar[r]_-{Tr\times Ts} & T\V\times T\V\ar[d]_{\xi\times\xi}\\
      && \V\times\V\ar[r]_\wedge &\V}
  \]
  the left hand side is a weak pullback, the middle diagram commutes, and in the
  right hand side we have ``lower path'' $\le$ ``upper path'' as indicated.
  Therefore, for such $\fw_1\in T(X\times X')$ and $\fw_2\in T(Y\times Y')$,
  there exists some $\fv\in T(X\times X'\times Y\times Y')$ which projects to
  $\fw\in T(X\times Y)$ and to
  $(\fw_1,\fw_2)\in T(X\times X')\times T(Y\times Y')$. Hence, taking also into
  account the definition of the $\V$-relation $T(r\owedge s)$,
  \begin{align*}
    Tr(\fx,\fx')\wedge Ts(\fy,\fy')
    &\le\bigvee\left\{\xi\cdot T\!(\wedge)\cdot T(r\times s)(\fv)\;\Bigl\lvert\;\fv\in T(X\times Y\times X'\times Y');
      \begin{aligned}
        & \fv\mapsto\fw \\
        & \fv\mapsto\fx',\fv\mapsto\fy'
      \end{aligned}\right\}\\
    &\le\bigvee\{T(r\owedge s)(\fw,\fw')\mid \fw'\in T(X'\times Y'),\;\can_{X',Y'}(\fw')=(\fx',\fy')\}.
  \end{align*}
\end{proof}

\begin{remark}
  We note that the inequality
  \[
    \xymatrix{T(\V\times\V)\ar[rr]^-{T(\wedge)}\ar[d]_{\langle\xi\cdot
        T\pi_1,\xi\cdot T\pi_2\rangle}\ar@{}[drr]|{\ge} && T\V\ar[d]^\xi\\
      \V\times\V\ar[rr]_-{\wedge} && \V}
  \]
  is always true.
\end{remark}

\begin{corollary}
  If the quantale $\V$ satisfies Assumption~\ref{d:ass:1}~\eqref{inj} and the
  diagrams
  \begin{equation*}
    \xymatrix{T(\V\times\V)\ar[rr]^-{T(\wedge)}\ar[d]_{\langle\xi\cdot
        T\pi_1,\xi\cdot T\pi_2\rangle} && T\V\ar[d]^\xi\\
      \V\times\V\ar[rr]_-{\wedge} && \V}
    \raisebox{-3ex}{\qquad\text{and}\qquad}
    \xymatrix{T1\ar[r]^{Tu}\ar[d]_{!} & T\V\ar[d]^{\xi} \\
      1\ar[r]_{u} & \V}
  \end{equation*}
  commute, for all $u\in\V$, then all Assumptions~\ref{d:ass:1} are satisfied.
\end{corollary}

Let $\mT$ be the ultrafilter monad $\mU=(U,m,e)$. Then, when $\V$ is any of the
quantales listed above but $\Delta$, all the needed conditions are
satisfied. Therefore, in particular we can conclude that:
\begin{examples}
  \begin{enumerate}
  \item \emph{The category $\Top=\Cats{(\mU,\two)}$ of topological spaces and
      continuous maps is weakly cartesian closed} (as shown by Rosick\'y in
    \cite{Ros99}).
  \item \emph{The category $\App=\Cats{(\mU,P_+)}$ of approach spaces and
      non-expansive maps is weakly cartesian closed.}
  \item In fact, for each continuous quantale structure on the lattice
    $([0,1],\le)\simeq ([0,\infty],\ge)$, $\Cats{(\mU,[0,1])}$ is weakly
    cartesian closed. In particular, \emph{the category of non-Archimedean
      approach spaces and non-expansive maps studied in \cite{CVO17} is weakly
      cartesian closed.}
  \item If $\V$ is a completely distributive complete lattice with
    $\otimes=\wedge$, then, with
    \[
      \xi: U\V\to\V,\,\fx\mapsto\bigwedge_{A\in\fx}\bigvee A,
    \]
    all the conditions of Theorem \ref{th:wcc} are satisfied (see
    \cite[Theorem~3.3]{Hof07}) and therefore $\Cats{(\mU,\V)}$ is weakly
    cartesian closed. In particular, with $\V=P\two$ being the powerset of a
    2-element set, we obtain that \emph{the category $\catfont{BiTop}$ of
      bitopological spaces and bicontinuous maps is weakly cartesian closed}
    (see \cite{HST14}).
  \end{enumerate}
\end{examples}

\begin{remark}
  For $\V=\Delta$ the quantale of distribution functions, we do not know whether
  there is an appropriate compact Hausdorff topology $\xi:U\V\to\V$ satisfying
  the conditions of this section.
\end{remark}

Now let $\mT$ be the free monoid monad $\mW=(W,m,e)$. For each quantale $\V$, we
consider
\[
  \xi: W\V\to\V,\,(v_1,\dots,v_n)\mapsto v_1\otimes\dots\otimes v_n,\,()\mapsto
  k
\]
which induces the extension $W:\Rels{\V}\to\Rels{\V}$ sending $r:X\relto Y$ to
the $\V$-relation $Wr:WX\relto WY$ given by
\[
  Wr((x_1,\dots,x_n),(y_1,\dots,y_m))=
  \begin{cases}
    r(x_1,y_1)\otimes\dots\otimes r(x_n,y_n) &\text{if $n=m$,}\\
    \bot & \text{if $n\neq m$}.
  \end{cases}
\]
The category $\Cats{(\mW,\two)}$ is equivalent to the category $\MultiOrd$ of
\emph{multi-ordered sets} and their morphisms (see~\cite{HST14}), more
generally, $(\mW,\V)$-categories can be interpreted as multi-$\V$-categories and
their morphisms. The representable multi-ordered sets are precisely the ordered
monoids, which is a special case of \cite{Her00,Her01} describing monoidal
categories as representable multi-categories (see also \cite{CCH15}). We recall
that the separated injective multi-ordered sets are precisely the quantales (see
\cite{LBKR12} and also \cite{Sea10}), and we conclude:
\begin{proposition}
  Every quantale is exponentiable in $\MultiOrd$.
\end{proposition}

\begin{theorem}
  If the quantale $\V$ is a frame, then $\Cats{(\mW,\V)}$ is weakly cartesian
  closed. In particular, $\MultiOrd$ is weakly cartesian closed.
\end{theorem}

Finally, for a monoid $(H,\cdot,h)$, we consider the monad
$\mH=(-\times H,m,e)$, with $m_X:X\times H\times H\to X\times H$ given by
$m_X(x,a,b)=(x,a\cdot b)$ and $e_X:X\to X\times H$ given by $e_X(x)=(x,h)$. Here
we consider
\[
  \xi:\V\times H\to\V,\,(v,a)\mapsto v,
\]
which leads to the extension $-\times H:\Rels{\V}\to\Rels{\V}$ sending the
$\V$-relation $r:X\relto Y$ to the $\V$-relation
$r\times H:X\times H\relto Y\times H$ with
\[
  r\times H((x,a),(y,b))=
  \begin{cases}
    r(x,y) & \text{if $a=b$,}\\
    \bot & \text{if $a\neq b$.}
  \end{cases}
\]
In particular, $(\mH,\two)$-categories can be interpreted as \emph{$H$-labelled
  ordered sets} and \emph{equivariant maps}.

For every quantale $\V$ and every $v:1\to\V$, the diagrams
\[
  \xymatrix{\V\times\V\times H\ar[r]^-{\wedge\times 1_H}\ar[d]_{\pi_{1,2}} &
    \V\times H\ar[d]^{\xi=\pi_1} \\ \V\times\V\ar[r]_{\wedge} & \V}
  \raisebox{-3ex}{\qquad\text{and}\qquad} \xymatrix{1\times H\ar[r]^{v\times
      1_H}\ar[d]_{!} & V\times H\ar[d]^{\xi} \\ 1\ar[r]_{v} & \V}
\]
commute, therefore we obtain:

\begin{theorem}
  For every quantale $\V$ satisfying Assumption~\ref{d:ass:1}~\eqref{inj}, the
  category $\Cats{(\mH,\V)}$ is weakly cartesian closed.
\end{theorem}



\begin{thebibliography}{LBKR12}

\bibitem[AHS90]{AHS90} Ji{\v{r}\'i} Ad\'{a}mek, Horst Herrlich, and
  George~E. Strecker.  \newblock \emph{Abstract and concrete categories: {T}he
    joy of cats}.  \newblock Pure and Applied Mathematics (New York). John Wiley
  \& Sons Inc., New York, 1990.  \newblock Republished in: Reprints in Theory
  and Applications of Categories, No. 17 (2006) pp.~1--507.

\bibitem[AR18]{AR18} Ji{{\v{r}}}{\'{i}} Ad{\'{a}}mek and Ji{{\v{r}}}{\'{i}}
  Rosick{\'{y}}.  \newblock How nice are free completions of categories?
  \newblock Technical report, 2018,
  \href{http://arxiv.org/abs/1806.02524}{{\ttfamily arXiv:1806.02524
      [math.CT]}}.

\bibitem[BBS04]{BBS04} Andrej Bauer, Lars Birkedal, and Dana~S. Scott.
  \newblock Equilogical spaces.  \newblock \emph{Theoretical Computer Science},
  315(1):35--59, 2004.

\bibitem[CCH15]{CCH15} Dimitri Chikhladze, Maria~Manuel Clementino, and Dirk
  Hofmann.  \newblock Representable $(\mathbb{T},\mathcal{V})$-categories.
  \newblock \emph{Applied Categorical Structures}, 23(6):829--858, January 2015,
  eprint: \url{http://www.mat.uc.pt/preprints/ps/p1247.pdf}.

\bibitem[CH03]{CH03} Maria~Manuel Clementino and Dirk Hofmann.  \newblock
  Topological features of lax algebras.  \newblock \emph{Applied Categorical
    Structures}, 11(3):267--286, June 2003, eprint:
  \url{http://www.mat.uc.pt/preprints/ps/p0109.ps}.

\bibitem[CH06]{CH06} Maria~Manuel Clementino and Dirk Hofmann.  \newblock
  Exponentiation in {$V$}-categories.  \newblock \emph{Topology and its
    Applications}, 153(16):3113--3128, October 2006.

\bibitem[CH09]{CH09} Maria~Manuel Clementino and Dirk Hofmann.  \newblock
  Lawvere completeness in topology.  \newblock \emph{Applied Categorical
    Structures}, 17(2):175--210, August 2009,
  \href{http://arxiv.org/abs/0704.3976}{{\ttfamily arXiv:0704.3976 [math.CT]}}.

\bibitem[CH17]{CH17} Maria~Manuel Clementino and Dirk Hofmann.  \newblock The
  {R}ise and {F}all of {$V$}-functors.  \newblock \emph{Fuzzy Sets and Systems},
  321:29--49, August 2017, eprint:
  \url{http://www.mat.uc.pt/preprints/ps/p1606.pdf}.

\bibitem[CHJ14]{CHJ14} Maria~Manuel Clementino, Dirk Hofmann, and George
  Janelidze.  \newblock The monads of classical algebra are seldom weakly
  cartesian.  \newblock \emph{Journal of Homotopy and Related Structures},
  9(1):175--197, November 2014, eprint:
  \url{http://www.mat.uc.pt/preprints/ps/p1246.pdf}.

\bibitem[CHT03]{CHT03} Maria~Manuel Clementino, Dirk Hofmann, and Walter Tholen.
  \newblock Exponentiability in categories of lax algebras.  \newblock
  \emph{Theory and Applications of Categories}, 11(15):337--352, 2003, eprint:
  \url{http://www.mat.uc.pt/preprints/ps/p0302.pdf}.

\bibitem[CT03]{CT03} Maria~Manuel Clementino and Walter Tholen.  \newblock
  Metric, topology and multicategory---a common approach.  \newblock
  \emph{Journal of Pure and Applied Algebra}, 179(1-2):13--47, April 2003.

\bibitem[CVO17]{CVO17} Eva Colebunders and Karen Van~Opdenbosch.  \newblock
  Topological properties of non-{A}rchimedean approach spaces.  \newblock
  \emph{Theory and Applications of Categories}, 32(41):1454--1484, 2017.

\bibitem[Fau55]{Fau55} William~M. Faucett.  \newblock Compact semigroups
  irreducibly connected between two idempotents.  \newblock \emph{Proceedings of
    the American Mathematical Society}, 6(5):741--747, May 1955.

\bibitem[Her00]{Her00} Claudio Hermida.  \newblock Representable
  multicategories.  \newblock \emph{Advances in Mathematics}, 151(2):164--225,
  May 2000.

\bibitem[Her01]{Her01} Claudio Hermida.  \newblock From coherent structures to
  universal properties.  \newblock \emph{Journal of Pure and Applied Algebra},
  165(1):7--61, November 2001.

\bibitem[Hof06]{Hof06} Dirk Hofmann.  \newblock Exponentiation for unitary
  structures.  \newblock \emph{Topology and its Applications},
  153(16):3180--3202, October 2006.

\bibitem[Hof07]{Hof07} Dirk Hofmann.  \newblock Topological theories and closed
  objects.  \newblock \emph{Advances in Mathematics}, 215(2):789--824, November
  2007.

\bibitem[Hof11]{Hof11} Dirk Hofmann.  \newblock Injective spaces via adjunction.
  \newblock \emph{Journal of Pure and Applied Algebra}, 215(3):283--302, March
  2011, \href{http://arxiv.org/abs/0804.0326}{{\ttfamily arXiv:0804.0326
      [math.CT]}}.

\bibitem[Hof14]{Hof14} Dirk Hofmann.  \newblock The enriched {V}ietoris monad on
  representable spaces.  \newblock \emph{Journal of Pure and Applied Algebra},
  218(12):2274--2318, December 2014,
  \href{http://arxiv.org/abs/1212.5539}{{\ttfamily arXiv:1212.5539 [math.CT]}}.

\bibitem[HR13]{HR13} Dirk Hofmann and Carla~D. Reis.  \newblock Probabilistic
  metric spaces as enriched categories.  \newblock \emph{Fuzzy Sets and
    Systems}, 210:1--21, January 2013,
  \href{http://arxiv.org/abs/1201.1161}{{\ttfamily arXiv:1201.1161 [math.GN]}}.

\bibitem[HS15]{HS15a} Dirk Hofmann and Gavin~J. Seal.  \newblock Exponentiable
  approach spaces.  \newblock \emph{Houston Journal of Mathematics},
  41(3):1051--1062, 2015, \href{http://arxiv.org/abs/1304.6862}{{\ttfamily
      arXiv:1304.6862 [math.GN]}}.

\bibitem[HST14]{HST14} Dirk Hofmann, Gavin~J. Seal, and Walter Tholen, editors.
  \newblock \emph{Monoidal {T}opology. {A} {C}ategorical {A}pproach to {O}rder,
    {M}etric, and {T}opology}, volume 153 of \emph{Encyclopedia of Mathematics
    and its Applications}.  \newblock Cambridge University Press, Cambridge,
  July 2014.  \newblock Authors: Maria Manuel Clementino, Eva Colebunders, Dirk
  Hofmann, Robert Lowen, Rory Lucyshyn-Wright, Gavin J.\ Seal and Walter Tholen.

\bibitem[HT10]{HT10} Dirk Hofmann and Walter Tholen.  \newblock {L}awvere
  completion and separation via closure.  \newblock \emph{Applied Categorical
    Structures}, 18(3):259--287, November 2010,
  \href{http://arxiv.org/abs/0801.0199}{{\ttfamily arXiv:0801.0199 [math.CT]}}.

\bibitem[Law73]{Law73} F.~William Lawvere.  \newblock Metric spaces, generalized
  logic, and closed categories.  \newblock \emph{Rendiconti del Seminario
    Matem\`{a}tico e Fisico di Milano}, 43(1):135--166, December 1973.
  \newblock {Republished in: Reprints in Theory and Applications of Categories,
    No. 1 (2002), 1--37}.

\bibitem[LBKR12]{LBKR12} Joachim Lambek, Michael Barr, John~F. Kennison, and
  Robert Raphael.  \newblock Injective hulls of partially ordered monoids.
  \newblock \emph{Theory and Applications of Categories}, 26(13):338--348, 2012.

\bibitem[Low97]{Low97} Robert Lowen.  \newblock \emph{Approach Spaces: The
    Missing Link in the Topology-Uniformity-Metric Triad}.  \newblock Oxford
  Mathematical Monographs. Oxford University Press, Oxford, 1997.

\bibitem[MS57]{MS57} Paul~S. Mostert and Allen~L. Shields.  \newblock On the
  structure of semi-groups on a compact manifold with boundary.  \newblock
  \emph{Annals of Mathematics. Second Series}, 65(1):117--143, January 1957.

\bibitem[MST06]{MST06} Mojgan Mahmoudi, Christoph Schubert, and Walter Tholen.
  \newblock Universality of coproducts in categories of lax algebras.  \newblock
  \emph{Applied Categorical Structures}, 14(3):243--249, June 2006.

\bibitem[Rib18]{Rib18} Willian Ribeiro.  \newblock On generalized equilogical
  spaces.  \newblock Technical Report 18-50, Department of Mathematics,
  University of Coimbra, 2018, \href{http://arxiv.org/abs/1811.08240}{{\ttfamily
      arXiv:1811.08240 [math.CT]}}.

\bibitem[Ros99]{Ros99} Ji{{\v{r}}\'{i}} Rosick{\'{y}}.  \newblock Cartesian
  closed exact completions.  \newblock \emph{Journal of Pure and Applied
    Algebra}, 142(3):261--270, October 1999.

\bibitem[Sea10]{Sea10} Gavin~J. Seal.  \newblock Order-adjoint monads and
  injective objects.  \newblock \emph{Journal of Pure and Applied Algebra},
  214(6):778--796, June 2010.

\bibitem[Tho09]{Tho09} Walter Tholen.  \newblock Ordered topological structures.
  \newblock \emph{Topology and its Applications}, 156(12):2148--2157, July 2009.

\end{thebibliography}

\end{document}